\documentclass[12pt]{article}
\usepackage{amsmath}
\usepackage{amssymb}
\usepackage{textcomp}
\usepackage{amsmath,amsthm,amsfonts,amssymb,amstext}
\usepackage{layout}
\newcommand{\K}{\ensuremath{\mathbb{K}}}
\newcommand{\N}{\ensuremath{\mathbb{N}}}
\newcommand{\Z}{\ensuremath{\mathbb{Z}}}

\newtheorem{thm}{Theorem}[section]
\newtheorem{cor}[thm]{Corollary}

\newtheorem{prop}[thm]{Proposition}
\newtheorem{lemma}[thm]{Lemma}

\newtheorem{note}[thm]{Note}
\newtheorem{definition}[thm]{Definition}
\newtheorem{assump}[thm]{Assumption}

\renewenvironment{proof}{\medskip\noindent{\emph {Proof:}\ }}{\qed \medskip
}

\begin{document}

\title{ \bf Tridiagonal pairs of $q$-Racah type,\\ the double lowering operator $\psi$, and the quantum algebra  
$U_q(\mathfrak{sl}_2)$}
\author{Sarah Bockting-Conrad
}

\date{}
\maketitle

\begin{abstract}
Let $\K$ denote an algebraically closed field and let $V$ denote a vector space over $\K$ with finite positive dimension.  
We consider an ordered pair of linear transformations $A: V\to V$ and $A^*: V \to V$ that satisfy the following four conditions:  
(i) Each of $A,A^*$ is diagonalizable;
(ii) there exists an ordering $\{ V_i\}_{i=0}^d$ of the eigenspaces of $A$ such that 
$A^* V_i \subseteq V_{i-1}+V_i+V_{i+1}
$ for $0\leq i\leq d$, where $V_{-1}=0$ and $V_{d+1}=0$;
(iii) there exists an ordering $\{V^*_i\}_{i=0}^{\delta}$ of the eigenspaces of $A^*$ such that 
$A V^*_i \subseteq V^*_{i-1}+V^*_i+V^*_{i+1}
$ for $0\leq i\leq \delta$, where $V^*_{-1}=0$ and $V^*_{\delta+1}=0$;
(iv) there does not exist a subspace $W$ of $V$ such that $AW \subseteq W$, $A^*W\subseteq W$, $W\neq 0$, $W\neq V$.  
We call such a pair a {\it tridiagonal pair} on $V$.  
It is known that $d=\delta$; to avoid trivialities assume $d\geq 1$.  
We assume that $A,A^*$ belongs to a family of tridiagonal pairs said to have 
$q$-Racah type.  This is the most general type of tridiagonal pair.  
Let $\{U_i\}_{i=0}^d$ and $\{U_i^\Downarrow\}_{i=0}^{d}$ denote the first and second split decompositions of $V$.  
In an earlier paper we introduced the double lowering operator $\psi:V\to V$.  
One feature of $\psi$ is that both 
$\psi U_i\subseteq U_{i-1}$ 
and 
$\psi U_i^\Downarrow\subseteq U_{i-1}^\Downarrow$ for $0\leq i\leq d$, where $U_{-1}=0$ and $U_{-1}^\Downarrow=0$.  
Define linear transformations $K:V\to V$ and $B:V\to V$ such that $(K-q^{d-2i}I)U_i=0$ and $(B-q^{d-2i}I)U_i^\Downarrow=0$ 
for $0\leq i\leq d$.  
Our results are summarized as follows.  
Using $\psi, K,B$ we obtain two actions of $U_q(\mathfrak{sl}_2)$ on $V$.  For each of these $U_q(\mathfrak{sl}_2)$-module structures, the Chevalley generator $e$ acts as a scalar multiple of $\psi$. 
For each of the $U_q(\mathfrak{sl}_2)$-module structures, we compute the action of the Casimir element on $V$.  We show that these two actions agree.  
Using this fact, we express $\psi$ as a rational function of $K^{\pm 1}, B^{\pm 1}$ in several ways.  Eliminating $\psi$ from these equations we find that $K$ and $B$ are related by a quadratic equation.

\end{abstract}

\section{Introduction}\label{section:intro}

Throughout this paper, $\K$ denotes an algebraically closed field.\\

We begin by recalling the notion of a tridiagonal pair.
We will use the following terms.  Let $V$ denote a vector space over $\K$ with finite positive dimension.  For a linear transformation $A:V\to V$ and a subspace $W\subseteq V$, we say that $W$ is an {\it eigenspace} of $A$ whenever $W\neq 0$ and there exists $\theta\in\K$ such that $W=\{ v\in V | Av=\theta v\}$.  In this case, $\theta$ is called the {\it eigenvalue} of $A$ associated with $W$.  We say that $A$ is {\it diagonalizable} whenever $V$ is spanned by the eigenspaces of $A$.

\begin{definition}  {\rm \cite[Definition 1.1]{Somealg}}. \label{def:tdp}
{\rm Let $V$ denote a vector space over $\K$ with finite positive dimension. 
By a {\em tridiagonal pair} (or {\em TD pair}) on $V$ we mean an ordered pair of linear 
transformations $A:V \to V$ and $A^*:V \to V$ that satisfy the following 
four conditions.
\begin{enumerate}
\item[{\rm (i)}] 
Each of $A,A^*$ is diagonalizable.
\item[{\rm (ii)}] 
There exists an ordering $\{V_i\}_{i=0}^d$ of the eigenspaces of $A$ 
such that 
\begin{equation}               \label{eq:t1}
A^* V_i \subseteq V_{i-1} + V_i+ V_{i+1} \qquad \qquad (0 \leq i \leq d),
\end{equation}
where $V_{-1} = 0$ and $V_{d+1}= 0$.
\item[{\rm (iii)}]
There exists an ordering $\{V^*_i\}_{i=0}^{\delta}$ of the eigenspaces of 
$A^*$ such that 
\begin{equation}                \label{eq:t2}
A V^*_i \subseteq V^*_{i-1} + V^*_i+ V^*_{i+1} 
\qquad \qquad (0 \leq i \leq \delta),
\end{equation}
where $V^*_{-1} = 0$ and $V^*_{\delta+1}= 0$.
\item [{\rm (iv)}]
There does not exist a subspace $W$ of $V$ such  that $AW\subseteq W$,
$A^*W\subseteq W$, $W\not=0$, $W\not=V$.
\end{enumerate}
We say the pair $A,A^*$ is {\em over $\K$}.}
\end{definition}

\begin{note}   \label{note:star}        \samepage
{\rm According to a common notational convention $A^*$ denotes 
the conjugate-transpose of $A$. We are not using this convention.
In a TD pair $A,A^*$ the linear transformations $A$ and $A^*$
are arbitrary subject to (i)--(iv) above.}
\end{note}

Referring to the TD pair in Definition \ref{def:tdp}, by  \cite[Lemma 4.5]{Somealg} 
the scalars 
 $d$ and $\delta $ are equal.
We call this common value the {\it diameter} of $A,A^*$.
To avoid trivialities, throughout this paper we assume that the diameter is at least one.\\

We now give some background on TD pairs; for more information we refer the reader to the survey \cite{Askeyscheme}.  
The notion of a TD pair originated in the theory of Q-polynomial distance-regular graphs \cite{subconstituent}.
The notion was formally introduced in \cite{Somealg}.  
Over time, connections were found between TD pairs and other areas of mathematics and physics.  
For instance, there are connections between TD pairs and  
representation theory \cite{TDfamily, neubauer, td-uqsl2, qtetalgebra, IT:qRacah, Koornwinder, Koornwinder2, aw}, orthogonal polynomials \cite{Askeyscheme, 2LT-PA}, partially ordered sets \cite{LPintro}, statistical mechanical models \cite{baseilhac, dolangrady, Onsager}, and other areas of physics \cite{LPcm, odake}.  
Among the above papers on representation theory, there are several works that connect TD pairs to quantum groups \cite{TDfamily, neubauer, td-uqsl2,   IT:qRacah}.  
These papers consider certain special classes of TD pairs.  
In \cite{TDfamily}, Curtin and Al-Najjar considered the class of mild TD pairs of $q$-Serre type.  They showed that these TD pairs induce a  
$U_q(\widehat{\mathfrak{sl}}\sb 2)$-module structure on their underlying vector space.
In \cite{neubauer}, Funk-Neubauer extended this construction to TD pairs of $q$-Hahn type.  
In \cite{td-uqsl2}, Ito and Terwilliger extended the construction to the entire $q$-Serre class.  
In \cite{IT:qRacah}, Ito and Terwilliger extended the construction to the $q$-Racah class.\\

In the present paper, we describe a relationship between TD pairs and quantum groups that appears to be new.  
In order to motivate our main results, we recall some basic facts concerning TD pairs.  
For the rest of this section, let $A,A^*$ denote a TD pair on $V$, as in Definition \ref{def:tdp}.  Fix an ordering $\{V_i\}_{i=0}^d$ (resp. $\{V_i^*\}_{i=0}^d$) of the eigenspaces of $A$ (resp. $A^*$) which satisfies (\ref{eq:t1}) (resp. (\ref{eq:t2})).  For $0\leq i\leq d$ let $\theta_i$ (resp. $\theta_i^*$) denote the eigenvalue of $A$ (resp. $A^*$) corresponding to $V_i$ (resp. $V_i^*$).   
By \cite[Theorem 11.1]{Somealg} the ratios
\begin{equation*}
\frac{\theta_{i-2}-\theta_{i+1}}{\theta_{i-1}-\theta_i},\qquad\qquad \frac{\theta_{i-2}^*-\theta_{i+1}^*}{\theta_{i-1}^*-\theta_i^*}\label{eq:intro-ratio}
\end{equation*}
are equal and independent of $i$ for $2\leq i\leq d-1$.  This gives two recurrence relations, whose solutions can be written in closed form.  
There are several cases \cite[Theorem 11.2]{Somealg}.  
The most general case is called the $q$-Racah case \cite[Section 1]{IT:qRacah}.  We will discuss this case shortly.\\

We now recall the split decompositions of $V$ \cite{Somealg}.  
For $0\leq i\leq d$ define
\begin{align*}
U_i&= (V^*_0+V^*_1+\cdots + V^*_i)\cap (V_i+V_{i+1}+\cdots + V_d),\\
U_i^{\Downarrow} &= (V^*_0+V^*_1+\cdots + V^*_i)\cap (V_0+V_1+\cdots + V_{d-i}).
\end{align*}
By \cite[Theorem 4.6]{Somealg},
both the sums $V=\sum_{i=0}^d U_i$ and $V=\sum_{i=0}^d U_i^{\Downarrow}$ are direct.  We call $\{U_i\}_{i=0}^d$ (resp. $\{U_i^{\Downarrow}\}_{i=0}^d$) the first split decomposition (resp.  second split decomposition) of $V$.  
By \cite[Theorem 4.6]{Somealg}, $A$ and $A^*$ act on the first split decomposition in the following way: 
\begin{align*}
&(A-\theta_i I)U_i\subseteq U_{i+1} &(0\leq i \leq d-1 ), \qquad &(A-\theta_d I)U_d=0,\\
&(A^*-\theta_i^* I)U_i\subseteq U_{i-1} &(1\leq i\leq d ), \qquad &(A^*-\theta_0^* I)U_0=0.
\end{align*}
By \cite[Theorem 4.6]{Somealg}, $A$ and $A^*$ act on the second split decomposition in the following way:
\begin{align*}
&(A-\theta_{d-i} I)U_i^{\Downarrow}\subseteq U_{i+1}^{\Downarrow}&(0\leq i\leq d-1 ), \qquad &(A-\theta_0 I)U_d^{\Downarrow}=0,\\
&(A^*-\theta_i^* I)U_i^{\Downarrow}\subseteq U_{i-1}^{\Downarrow}&(1\leq i\leq d ), \qquad &(A^*-\theta_0^* I)U_0^{\Downarrow}=0.
\end{align*}

\bigskip

We now describe the $q$-Racah case.  
We say that the TD pair $A,A^*$ has {\it $q$-Racah type} whenever
 there exist nonzero scalars $q,a,b\in\K$ 
 such that $q^4\neq 1$ and 
\begin{equation*}
\theta_i=aq^{d-2i}+a^{-1}q^{2i-d},\qquad\qquad
\theta_i^*=bq^{d-2i}+b^{-1}q^{2i-d}
\end{equation*}
for $0\leq i\leq d$.\\

For the rest of this section assume that $A,A^*$ has $q$-Racah type.\\

We recall the maps $K$ and $B$ \cite[Section 1.1]{augTDalg}.  
Let $K:V\to V$ denote the linear transformation such that for $0\leq i \leq d$, $U_i$ is an eigenspace of $K$ with eigenvalue $q^{d-2i}$.  Let $B:V\to V$ denote the linear transformation such that for $0\leq i \leq d$, $U_i^{\Downarrow}$ is an eigenspace of $B$ with eigenvalue $q^{d-2i}$.  
In \cite[Section 1.1]{augTDalg} it is shown how each of $K, B$ is related to 
$A$ and $A^*$,   
but the relationship between $K$ and $B$ is not discussed.  
One of the goals of the present paper is to describe 
how $K$ and $B$ are related.  
We show that 
\begin{align}
aK^2-\frac{a^{-1}q-aq^{-1}}{q-q^{-1}}\ KB&-\frac{aq-a^{-1}q^{-1}}{q-q^{-1}}\ BK+a^{-1}B^2=0.\label{eq:intro1}
\end{align}

We now recall the raising maps 
$R:V\to V$ and $R^\Downarrow:V\to V$ \cite[Definition 6.1]{Somealg}.   
Following \cite{augTDalg} we define
\begin{align*}
R = A-aK-a^{-1}K^{-1},\quad\qquad 
R^{\Downarrow} = A-a^{-1}B-aB^{-1}.
\end{align*}
As we will see in Section \ref{section:U}, for $0\leq i\leq d$, the map $R$ acts on $U_i$ as 
$A-\theta_i I$, and the map $R^\Downarrow$ acts on $U_i^\Downarrow$ as 
$A-\theta_{d-i} I$.  
Moreover
\begin{align*}
RU_i&\subseteq U_{i+1}  &(0\leq i\leq d-1), \qquad\qquad &RU_d=0, \\
 R^\Downarrow U_i^\Downarrow&\subseteq U_{i+1}^\Downarrow  &(0\leq i\leq d-1), \qquad\qquad &R^\Downarrow U_d^\Downarrow=0.
 \end{align*}

We now bring in the linear transformation $\Psi:V\to V$ \cite[Lemma 11.1]{bockting}.  
We will work with the normalization $\psi=(q-q^{-1})(q^d-q^{-d})\Psi$.  One attraction of $\psi$ is that by 
\cite[Lemma 11.2, Corollary 15.3]{bockting}, 
\begin{equation}
\psi U_i\subseteq U_{i-1},\qquad\qquad \psi U_i^\Downarrow\subseteq U_{i-1}^\Downarrow
\end{equation}
for $1\leq i\leq d$ and both $\psi U_0=0$ and $\psi U_0^\Downarrow=0$.  
Drawing on the results in \cite{bockting}, 
we obtain some equations that link $\psi$ to the maps $K,B,R,R^\Downarrow$.  
From these equations we obtain 
two $U_q(\mathfrak{sl}_2)$-module structures on $V$.  
For the first $U_q(\mathfrak{sl}_2)$-module structure, the Chevalley generators $e,f,k,k^{-1}$ act as follows:
\begin{center}
\begin{tabular}{c|cccc}
{\rm element of $U_q(\mathfrak{sl}_2)$} &  $e$ & $f$ &$k$ & $k^{-1}$
\\
\hline
{\rm action on $V$} & $(q-q^{-1})^{-1}\psi$ \ \ & $(q-q^{-1})^{-1}R$ \ \ &$K$ \ & $K^{-1}$
\end{tabular}
\end{center}
For the second $U_q(\mathfrak{sl}_2)$-module structure, the Chevalley generators act as follows:
\begin{center}
\begin{tabular}{c|cccc}
{\rm element of $U_q(\mathfrak{sl}_2)$} &  $e$ & $f$ &$k$ & $k^{-1}$
\\
\hline
{\rm action on $V$} & $(q-q^{-1})^{-1}\psi$ \ \ & $(q-q^{-1})^{-1}R^\Downarrow$ \ \ &$B$ \ & $B^{-1}$
\end{tabular}
\end{center}
 For each of the above $U_q(\mathfrak{sl}_2)$-module structures we obtain a direct sum decomposition of $V$ into irreducible $U_q(\mathfrak{sl}_2)$-submodules.  
Also, in each case, we compute the action of the Casimir element on $V$.  We show that these two actions agree.
Using this information  
 we show that $\psi$ is equal to each of the following:
\begin{eqnarray}
&{\displaystyle \frac{I-BK^{-1}}{q(aI-a^{-1}BK^{-1})}, \qquad\qquad\qquad \frac{I-KB^{-1}}{q(a^{-1}I-aKB^{-1})}},\label{eq:intro-psiequations1}\\ \medskip
&{\displaystyle \frac{q(I-K^{-1}B)}{aI-a^{-1}K^{-1}B},\qquad\qquad\qquad \frac{q(I-B^{-1}K)}{a^{-1}I-aB^{-1}K}}.
\label{eq:intro-psiequations2}
\end{eqnarray}
 Line (\ref{eq:intro1}) 
 is a consequence of the fact that the four expressions in (\ref{eq:intro-psiequations1}), (\ref{eq:intro-psiequations2}) are equal. \\ 
 
We have now finished summarizing the main results of the paper.  As we prove these results, we obtain some secondary results that might be of independent interest.  These results are about how the two $U_q(\mathfrak{sl}_2)$-module structures on $V$ are related.  We also comment on the relationship between $A$ and $\psi$.\\   

The paper is organized as follows.  In Section \ref{section:prelim} we discuss some preliminary facts concerning TD pairs and TD systems.  
In Sections \ref{section:U} and \ref{section:Urefine} we discuss the split decompositions of $V$ as well as the maps $K$ and $B$.  
In Section \ref{section:psi} we discuss the map $\psi$.
In Section \ref{section:uqsl2} we recall the algebra $U_q(\mathfrak{sl}_2)$ and its modules.  
In Section \ref{section:Umod} we use our TD pair to obtain a $U_q(\mathfrak{sl}_2)$-module structure on $V$.
In Section \ref{section:Umoddd} we use our TD pair to obtain a second $U_q(\mathfrak{sl}_2)$-module structure on $V$.
In Section \ref{section:related} we discuss how the two  $U_q(\mathfrak{sl}_2)$-actions on $V$ are related.  This leads to some equations relating the maps $\psi, K^{\pm 1}, B^{\pm 1}$.  These equations are discussed further in Section \ref{section:related2}. 
In Section \ref{section:Apsi} we discuss how $A,\psi$ are related.\\

 In Lemma \ref{lemma:Uaction} and Lemma \ref{lemma:Uddaction} we give two $U_q(\mathfrak{sl}_2)$-actions on $V$.  As far as we know, these actions are unrelated to the actions given in earlier papers on 
 TD pairs and quantum groups \cite{TDfamily, neubauer, td-uqsl2, IT:qRacah}.  This is discussed in Note \ref{note:uqsl2hat}.

%%%%%%%%%%%%%%%
%%%%%%%%%%%%%%%
%%%%%%%%%%%%%%%
\section{Preliminaries}\label{section:prelim}
%%%%%%%%%%%%%%%
%%%%%%%%%%%%%%%
%%%%%%%%%%%%%%%

When working with a tridiagonal pair, it is useful to consider a closely related object called a tridiagonal system. In order to define this object, we first recall some facts from elementary linear algebra \cite[Section 2]{Somealg}.\\

We use the following conventions.  When we discuss an algebra, we mean a unital associative algebra. When we discuss a subalgebra, we assume that it has the same unit as the parent algebra.\\

Let $V$ denote a vector space over $\K$ with finite positive dimension.
Let $ {\rm End} (V)$ denote the $\K$-algebra consisting of all linear transformations from $V$ to $V$.
Let $A$ denote a diagonalizable element in ${\rm End}(V)$.  
Let $\{V_i\}_{i=0}^d$ denote an ordering of the eigenspaces of $A$.
For $0\leq i\leq d$ let $\theta_i$ be the eigenvalue of $A$ corresponding to $V_i$.
Define $E_i\in {\rm End}(V)$ by
$ (E_i-I)V_i=0$ and 
$E_iV_j=0$ if $j\neq i$  $(0\leq j\leq d)$.
In other words, $E_i$ is the projection map from $V$ onto $V_i$.  
We refer to $E_i$ as the {\it primitive idempotent} of
$A$ associated with $\theta_i$.
By elementary linear algebra, the following (i)--(iv) hold: (i)
$AE_i = E_iA = \theta_iE_i$   $(0 \leq i \leq d)$;
(ii) $E_iE_j = \delta_{ij}E_i$   $(0 \leq i,j\leq d)$;
(iii) $V_i=E_iV$   $(0 \leq i \leq d)$;
(iv) $I=\sum_{i=0}^d E_i$.
Moreover 
\begin{eqnarray*}
E_i = \prod_{{0 \leq  j \leq d}\atop{j\not=i}} {{A-\theta_j I}\over {\theta_i-\theta_j}}\qquad \qquad (0 \leq i \leq d).\label{EA} 
\end{eqnarray*}

Let $M$ denote the subalgebra of ${\rm End}(V)$ generated by $A$.  
Note that each of $\{ A^i\}_{i=0}^d$, $\{E_i\}_{i=0}^d$ is a basis for the $\K$-vector space $M$.\\  

Let $A,A^*$ denote a TD pair on $V$.  An ordering of the eigenspaces of $A$ (resp. $A^*$) is said to be {\it standard} whenever it satisfies (\ref{eq:t1}) (resp. (\ref{eq:t2})).  
Let $\{V_i\}_{i=0}^d$ denote a standard ordering of the eigenspaces of $A$.  By \cite[Lemma 2.4]{Somealg}, the ordering $\{V_{d-i}\}_{i=0}^d$ is standard and no further ordering of the eigenspaces of $A$ is standard.  A similar result holds for the eigenspaces of $A^*$.  An ordering of the primitive idempotents of $A$ (resp. $A^*$) is said to be {\it standard} whenever the corresponding ordering of the eigenspaces of $A$ (resp. $A^*$) is standard.

\begin{definition}{\rm \cite[Definition 2.1]{TDclass}.} \label{def:TDsys} {\rm
Let $V$ denote a vector space over $\K$ with finite positive dimension.  
By a {\it tridiagonal system} (or {\em TD system}) on $V,$  we mean a 
sequence 
\begin{equation*}
\Phi = (A; \{ E_i\}_{i=0}^d; A^*;\{ E_i^*\}_{i=0}^d)
\end{equation*}
 that satisfies  (i)--(iii) below.
\begin{enumerate}
\item[{\rm (i)}] $A, A^*$ is a tridiagonal pair on $V$. 
\item [{\rm (ii)}]$\{ E_i\}_{i=0}^d$ is a standard ordering of the primitive 
idempotents of $\;A$.
\item[{\rm (iii)}] $\{ E_i^*\}_{i=0}^d$ is a standard ordering of the primitive 
idempotents of $\;A^*$.
\end{enumerate}
We call $d$ the {\it diameter} of $\Phi$,  
and say $\Phi$ is {\it over } $\K$. 
 For notational convenience, set $E_{-1}=0$, $E_{d+1}=0$, $
E^*_{-1}=0$, $E^*_{d+1}=0$.}
\end{definition}

In Definition \ref{def:TDsys} we do not assume that the primitive idempotents $\{ E_i\}_{i=0}^d,\{ E_i^*\}_{i=0}^d$ all 
have rank 1.  A TD system for which each of these primitive idempotents has rank 1 is called a Leonard system \cite{2LT}.
The Leonard systems  are classified up to isomorphism \cite[Theorem 1.9]{2LT}.\\

For the rest of this section, fix a TD system $\Phi$ on $V$ as in Definition \ref{def:TDsys}.

\begin{definition}\label{def:main}{\rm
For $0\leq i\leq d$ let $\theta_i$ (resp. $\theta_i^*$) denote the eigenvalue of $A$ (resp. $A^*$) associated with $E_i$ (resp. $E_i^*$).
We refer to $\{\theta_i\}_{i=0}^d$ (resp. $\{\theta_i^*\}_{i=0}^d$) as the {\it eigenvalue sequence} (resp. {\it dual eigenvalue sequence}) of $\Phi$.  
By construction $\{\theta_i\}_{i=0}^d$ are mutually distinct and $\{\theta_i^*\}_{i=0}^d$ are mutually distinct.
}\end{definition}

Our TD system $\Phi$ can be modified in a number of ways to get a new TD system \cite[Section 3]{Somealg}.  For example, the sequence
\begin{eqnarray*}
\Phi^{\Downarrow} = (A;\{ E_{d-i}\}_{i=0}^d;A^*;\{ E_i^*\}_{i=0}^d)
\end{eqnarray*}
is a TD system on $V$. 
Following  {\rm \cite[Section 3]{Somealg}}, we call $\Phi^{\Downarrow}$ the {\it second inversion} of $\Phi$.
When discussing $\Phi^{\Downarrow}$, we use the following notational convention.
For any object $f$ associated with $\Phi$, let $f^{\Downarrow}$ denote the corresponding object associated with $\Phi^{\Downarrow}$.\\

We associate with $\Phi$ a family of polynomials as follows.
Let $x$ be an indeterminate.  Let $\K [x]$ denote the $\K$-algebra consisting of the polynomials in $x$ that have all coefficients in $\K$.  For $0\leq i\leq j\leq d+1$, we define the polynomials $\tau_{ij}=\tau_{ij}(\Phi)$ in $\K [x]$ by
\begin{align} \tau_{ij}&=(x-\theta_i)(x-\theta_{i+1})\cdot\cdot\cdot (x-\theta_{j-1}). \label{eq:tau}
\end{align}
Note that $\tau_{ij}$ is monic with degree $j-i$.  In particular, $\tau_{ii}=1$.
For notational convenience, define $\tau_{i,i-1}=0$.

\begin{assump}\label{assump:main}
Let $\Phi$ denote a TD system on $V$ as in Definition \ref{def:TDsys}.  
Assume that $\Phi$ has $q$-Racah type. 
Thus, there exist nonzero $q,a,b\in\K$ 
 such that $q^4\neq 1$ and 
\begin{equation}
\theta_i=aq^{d-2i}+a^{-1}q^{2i-d},\qquad\qquad
\theta_i^*=bq^{d-2i}+b^{-1}q^{2i-d}
\label{eq:theta-assump}
\end{equation}
for $0\leq i\leq d$.  
\end{assump}

\begin{lemma}\label{note:main}
With reference to Assumption \ref{assump:main}, the following hold.
\begin{itemize}
\item[{\rm (i)}] Neither of $a^2$, $b^2$ is among $q^{2d-2},q^{2d-4},\mathellipsis, q^{2-2d}$.
\item[{\rm (ii)}] $q^{2i}\neq 1$ for $1\leq i\leq d$.
\end{itemize}

\end{lemma}
\begin{proof}
By Definition \ref{def:main} 
the $\{\theta_i\}_{i=0}^d$ are mutually distinct and
the $\{\theta_i^*\}_{i=0}^d$ are mutually distinct.    
\end{proof}

%%%%%%%%%%%%%%%%%%%
%%%%%%%%%%%%%%%%%%%
%%%%%%%%%%%%%%%%%%%
\section{The first and second split decomposition of $V$}\label{section:U}
%%%%%%%%%%%%%%%%%%%
%%%%%%%%%%%%%%%%%%%
%%%%%%%%%%%%%%%%%%%

We continue to discuss the situation of Assumption \ref{assump:main}.   
We use the following concept.  
By a {\it decomposition} of $V,$ we mean a sequence of subspaces whose direct sum is $V$.  
For example, $\{ E_iV\}_{i=0}^d$ and $\{ E_i^*V\}_{i=0}^d$ are decompositions of $V$.
In this section we consider two other decompositions of $V$ called the first and second split decomposition.\\    

For $0\leq i\leq d$ define the subspace $U_i \subseteq V$ by
\begin{align*}
U_i = (E^*_0V+E^*_1V+\cdots + E^*_iV)\cap (E_iV+E_{i+1}V+\cdots + E_dV).
\end{align*}
For notational convenience, define $U_{-1}=0$ and $U_{d+1}=0$.
We note that 
\begin{equation*}
U_i^{\Downarrow}=(E_0^*V+E_1^*V +\cdots + E_i^* V)\cap (E_0V+E_1V+\cdots + E_{d-i} V).
\end{equation*}

By \cite[Theorem 4.6]{Somealg},  
the sequence $\{U_i\}_{i=0}^d$  (resp. $\{U_i^\Downarrow\}_{i=0}^d$) is a decomposition of $V$.  We refer to $\{U_i\}_{i=0}^d$  (resp. $\{U_i^\Downarrow\}_{i=0}^d$) as the {\it first split decomposition} (resp. {\it second split decomposition}) of $V$ with respect to $\Phi$.  
By \cite[Corollary 5.7]{Somealg}, for $0\leq i\leq d$ the dimensions of
$E_iV$, $E_i^*V$, $U_i$, $U_i^\Downarrow$ coincide.  
By \cite[Lemma 4.1]{bockting}, 
\begin{equation}
U_0+U_1+\cdots +U_i=U_0^\Downarrow+U_1^\Downarrow+\cdots +U_i^\Downarrow \qquad\qquad (0\leq i\leq d).\label{eq:UUdd}\\
\end{equation}

\begin{definition}\label{def:KB}{\rm
Define $K\in {\rm End}(V)$ such that for $0\leq i\leq d$, $U_i$ is the eigenspace of $K$ with eigenvalue $q^{d-2i}$.  In other words,
\begin{align}
(K-q^{d-2i}I)U_i&=0\qquad\qquad (0\leq i\leq d).\label{eq:Kdef}
\end{align}
}\end{definition}

\begin{definition}\label{def:KB-B}{\rm
Define $B\in {\rm End}(V)$ such that for $0\leq i\leq d$, $U_i^{\Downarrow}$ is the eigenspace of $B$ with eigenvalue $q^{d-2i}$.  In other words,
\begin{align}
(B-q^{d-2i}I)U_i^\Downarrow&=0\qquad \qquad (0\leq i\leq d).\label{eq:Bdef}
\end{align}
Observe that $B=K^{\Downarrow}$.  
}\end{definition}

\bigskip

By construction each of $K, B$ is invertible and diagonalizable on $V$.  

\begin{lemma}\label{lemma:KUdd}
For $0\leq i\leq d$,
\begin{align}
(B-q^{d-2i}I)U_i&\subseteq U_0+U_1+\cdots +U_{i-1},\label{eq:BU}\\
(K-q^{d-2i}I)U_i^\Downarrow &\subseteq U_0^\Downarrow+U_1^\Downarrow+\cdots +U_{i-1}^\Downarrow.\label{eq:KUdd}
\end{align}
\end{lemma}
\begin{proof}
We first show (\ref{eq:BU}).  By (\ref{eq:UUdd}), $U_i\subseteq U_0^\Downarrow+U_1^\Downarrow+\cdots +U_i^\Downarrow$.  Use this fact along with (\ref{eq:Bdef}).

The proof of (\ref{eq:KUdd}) is similar.
\end{proof}\\

Following \cite{augTDalg} define
\begin{equation}
R = A-aK-a^{-1}K^{-1}.\label{eq:RAK}
\end{equation}
Note that
\begin{equation}
R^{\Downarrow} = A-a^{-1}B-aB^{-1}.\label{eq:RAKdd}
\end{equation}

We remark that the map $R$ from (\ref{eq:RAK}) is the same as the map $R$ given in \cite[Definition 6.1]{Somealg}.\\  

We now recall some results concerning $R$ and $R^\Downarrow$.  

\begin{lemma}{\rm \cite[Lemma 6.2]{Somealg}.}\label{lemma:RonU}
For $0\leq i\leq d$, the map $R$ acts on $U_i$ as $A-\theta_iI$, and the map $R^\Downarrow$ acts on $U_i^\Downarrow$ as $A-\theta_{d-i}I$.  
\end{lemma}
\bigskip

By  \cite[Corollary 6.3]{Somealg}, 
we have
\begin{align}
RU_i&\subseteq U_{i+1}  &(0\leq i\leq d-1), \qquad\qquad &RU_d=0, \label{RU}\\
 R^\Downarrow U_i^\Downarrow&\subseteq U_{i+1}^\Downarrow  &(0\leq i\leq d-1), \qquad\qquad &R^\Downarrow U_d^\Downarrow=0.\nonumber 
 \end{align}
Moreover $R^{d+1}=0$ and $(R^\Downarrow)^{d+1}=0$.  
In light of these comments, we
 refer to $R$ (resp. $R^\Downarrow$) as the {\it raising map} for $\Phi$ (resp. $\Phi^\Downarrow$). \\

We now recall some results concerning $K$ and $B$.  
\begin{lemma}{\rm \cite[Section 1.1]{augTDalg}}.\label{lemma:KRKinv}
Both
\begin{equation}
KRK^{-1}=q^{-2}R,\qquad\qquad BR^{\Downarrow}B^{-1}=q^{-2}R^\Downarrow.\label{eq:KRKinv}
\end{equation}
\end{lemma}
\begin{proof}
We first show the equation on the left in (\ref{eq:KRKinv}).  Recall that for $0\leq i\leq d$, $U_i$ is an eigenspace for $K$ with eigenvalue $q^{d-2i}$.  Use this fact along with (\ref{RU}).

The proof is similar for the equation on the right in (\ref{eq:KRKinv}).
\end{proof}

\begin{lemma}{\rm \cite[Section 1.1]{augTDalg}}.\label{lemma:AKq-Weyl}
Both
\begin{equation}
\frac{qKA-q^{-1}AK}{q-q^{-1}}=aK^2+a^{-1}I,\qquad\qquad 
\frac{qBA-q^{-1}AB}{q-q^{-1}}=a^{-1}B^2+aI. \label{eq:AKq-Weyl}
\end{equation}
\end{lemma}
\begin{proof}
First we show the equation on the left in (\ref{eq:AKq-Weyl}).  
By Lemma \ref{lemma:KRKinv}, $qKR-q^{-1}RK=0$.  In this equation, eliminate $R$ using (\ref{eq:RAK}).

The proof is similar for the equation on the right in (\ref{eq:AKq-Weyl}).
\end{proof}

We conclude this section by giving a result which relates $R$ and $R^\Downarrow$.  Combining (\ref{eq:RAK}), (\ref{eq:RAKdd}) we obtain 
\begin{equation}
R^{\Downarrow}-R=a K+a^{-1}K^{-1}-a^{-1}B-aB^{-1}.\label{eq:Rdiff}
\end{equation}

%%%%%%%%%%%%%%%%%%%
%%%%%%%%%%%%%%%%%%%
%%%%%%%%%%%%%%%%%%%
\section{Comments on the split decompositions of $V$}\label{section:Urefine}
%%%%%%%%%%%%%%%%%%%
%%%%%%%%%%%%%%%%%%%
%%%%%%%%%%%%%%%%%%%

We continue to discuss the situation of Assumption \ref{assump:main}.  
In Section \ref{section:U} we recalled the first and second split decomposition of $V$.  In this section we collect a few related facts for later use. 
   
\begin{definition}{\rm \cite[Definition 6.1]{bockting}. \label{def:K_i}
 For $0\leq i\leq d/2$, define the subspace $K_i\subseteq V$ by
\begin{equation*}
K_i=(E_0^*V+E_1^*V+\cdots+E_i^*V)\cap(E_iV+E_{i+1}V+\cdots+E_{d-i}V).
\end{equation*}
Observe that $K_0=E_0^*V=U_0$ and $K_i^\Downarrow=K_i$.  Also note that $K_i=U_i\cap U_i^\Downarrow$.
}\end{definition}

\begin{lemma}{\rm \cite[Theorem 4.7]{refinement}}.\label{lemma:refinement}
For $0\leq j\leq d$ the following sum is direct:
\begin{equation*}
U_j=\sum_{i=0}^{min\{j,d-j\}} \tau_{ij}(A) K_i.
\end{equation*} 
\end{lemma}

\begin{cor} {\rm \cite[Theorem 4.8]{refinement}}. \label{cor:Urefine}
The following sum is direct:
\begin{eqnarray*}
V=\sum_{i=0}^{\lfloor d/2\rfloor}\sum_{j=i}^{d-i} \tau_{ij}(A)K_i.\label{eq:Urefine1}
\end{eqnarray*}
\end{cor}

The refinement of the first split decomposition given above yields the following description of the kernel of the map $R$ from Section \ref{section:U}.

\begin{lemma}\label{lemma:Rkernel}
For $0\leq i< d/2$, the restriction of $R$ to $U_i$ is injective.  
For $d/2\leq i\leq d$, the restriction of $R$ to $U_{i}$ is surjective with kernel $\tau_{d-i,i}(A)K_{d-i}$.  
Moreover the kernel of $R$ on $V$ is $\sum_{i=0}^{\lfloor d/2\rfloor} \tau_{i,d-i}(A)K_i$.
\end{lemma}
\begin{proof}
The claims concerning injectivity and surjectivity follow from \cite[Lemma 6.5]{Somealg}.
Let $d/2\leq i\leq d$.  We now show that the kernel of $R$ on $U_i$ is $\tau_{d-i,i}(A)K_{d-i}$. 
Recall that $R$ acts on $U_i$ as $A-\theta_i I$ and $RU_i\subseteq U_{i+1}$.  
The result follows from this along with (\ref{eq:tau}) and Lemma \ref{lemma:refinement}.
\end{proof}\\

Let $M$ denote the subalgebra of ${\rm End}(V)$ generated by $A$.
For the rest of this section, we view $V$ as an $M$-module.  
From this point of view, for $0\leq i\leq d/2$, $MK_i$ is  
the $M$-submodule of $V$ generated by $K_i$.

\begin{lemma}{\rm \cite[Lemma 8.2]{bockting}}.\label{lemma:MK}
For $0\leq i\leq d/2$ such that $K_i\neq 0$, the sum
\begin{equation}
MK_i=K_i+\tau_{i,i+1}(A)K_i+\tau_{i,i+2}(A)K_i+\cdots+\tau_{i,d-i}(A)K_i\label{eq:MK1}
\end{equation}
is direct.  Moreover $\tau_{i,d-i+1}$ is the minimal polynomial for the action of $A$ on $MK_i$.
\end{lemma}

\begin{cor}{\rm \cite[Corollary 8.3]{bockting}}.\label{cor:Mv1}
For $0\leq i\leq d/2$ and $0\neq v\in K_i$, the vector space $Mv$ has a basis
\begin{equation*}
v,\quad \tau_{i,i+1}(A)v,\quad \tau_{i,i+2}(A) v,\quad \mathellipsis,\quad \tau_{i,d-i}(A)v.
\end{equation*}
\end{cor}

\begin{prop}{\rm \cite[Lemma 8.4]{bockting}}.\label{prop:VMK_i}  
The following is a direct sum of $M$-modules:
\begin{eqnarray}
&V = \displaystyle\sum_{i=0}^{\lfloor d/2\rfloor} MK_i.\label{eq:VMK_i}
\end{eqnarray}
\end{prop}

%%%%%%%%%%%%%%%%%%%%%%%%%%%%%%%%%%
%%%%%%%%%%%%%%%%%%%%%%%%%%%%%%%%%%
%%%%%%%%%%%%%%%%%%%%%%%%%%%%%%%%%%
\section{The linear transformation $\psi$}\label{section:psi}
%%%%%%%%%%%%%%%%%%%%%%%%%%%%%%%%%%
%%%%%%%%%%%%%%%%%%%%%%%%%%%%%%%%%%
%%%%%%%%%%%%%%%%%%%%%%%%%%%%%%%%%%

We continue to discuss the situation of Assumption \ref{assump:main}.   
In \cite[Section 11]{bockting} we introduced an element $\Psi\in {\rm End}(V)$.  
For our present purpose it is convenient to use the normalization $\psi=(q-q^{-1})(q^d-q^{-d})\Psi$.

\begin{lemma}{\rm \cite[Lemma 11.7]{bockting}}.\label{lemma:Rpsi}\label{lemma:psidef'}
The map $\psi$ is the unique element of ${\rm End}(V)$ such that both 
\begin{equation}
\psi R - R\psi = (q-q^{-1})(K-K^{-1}) \label{psiR}
\end{equation}
and
$\psi K_i=0$ for $0\leq i\leq d/2$.
\end{lemma}

We now clarify the meaning of $\psi$. Recall the decomposition of $V$ given in Corollary \ref{cor:Urefine} and consider the summand $\tau_{ij}(A)K_i$.  We describe the action of $\psi$ on this summand.  By \cite[Equation (58)]{bockting}, for $v\in K_i$,
\begin{equation}
\psi\tau_{ij}(A)v=(q^{j-i}-q^{i-j})(q^{d-i-j+1}-q^{i+j-d-1})\tau_{i,j-1}(A)v.\label{eq:psiclar}
\end{equation}
  We note that the following hold on $\tau_{ij}(A)K_i$:
\begin{align*}
R\psi&=(q^{j-i}-q^{i-j})(q^{d-i-j+1}-q^{i+j-d-1})I,\\
\psi R&=(q^{j-i+1}-q^{i-j-1})(q^{d-i-j}-q^{i+j-d})I.
\end{align*}

\begin{lemma}{\rm \cite[Corollary 15.2]{bockting}}.\label{lemma:psiequal}
With reference to Lemma \ref{lemma:psidef'},
$\psi^{\Downarrow}=\psi$.
\end{lemma}

\begin{lemma}{\rm \cite[Lemma 11.2]{bockting}}. \label{lemma:psiU}
With reference to Lemma \ref{lemma:psidef'},
\begin{equation}
\psi U_i\subseteq U_{i-1},\qquad\qquad \psi U_i^\Downarrow\subseteq U_{i-1}^\Downarrow
\end{equation}
for $1\leq i\leq d$ and both $\psi U_0=0$ and $\psi U_0^\Downarrow=0$.   
Moreover $\psi^{d+1}=0$.
\end{lemma}

\begin{lemma}\label{lemma:KpsiKinv}
With reference to Lemma \ref{lemma:psidef'},
\begin{equation}
K\psi K^{-1}=q^2\psi, \qquad\qquad B\psi B^{-1}=q^2\psi.\label{eq:KpsiKinv}
\end{equation}
\end{lemma}
\begin{proof}
Use Lemma \ref{lemma:psiU} together with the definitions of $K$ and $B$. 
\end{proof}

\begin{lemma}\label{lemma:psiUkernel}
For $0\leq i\leq d/2$, $K_i$ is the kernel of $\psi$ acting on $U_i$.
\end{lemma}
\begin{proof}
Use Lemma \ref{lemma:refinement} along with (\ref{eq:psiclar}) and the fact that $q^{2j}\neq 1$ for $1\leq j\leq d$.
\end{proof}

\bigskip

In Lemma \ref{lemma:psidef'} we gave a characterization of $\psi$.  Shortly we will give a second characterization. That characterization will be based on the following result.

\begin{lemma}\label{lemma:X}
Given $X\in{\rm End}(V)$ such that $XR=RX$ and $XU_i\subseteq U_{i-1}$ for $0\leq i\leq d$.
Then $X=0$.
\end{lemma}
\begin{proof}
By  (\ref{eq:tau}), Lemma \ref{lemma:RonU}, and Corollary \ref{cor:Urefine}, it suffices to show that $XR^hK_i=0$ for $0\leq i\leq d/2$ and $0\leq h\leq d-2i$.  
Since $XR=RX$, it suffices to show that $XK_i=0$ for $0\leq i\leq d/2$.  
Let $i$ be given.
First assume that $i=0$.
Then $XK_0=0$, since $K_0=U_0$ and $XU_0=0$.  
Next assume that $i\geq 1$.  
 By Lemma \ref{lemma:Rkernel} and $R^{d-2i}K_i=\tau_{i,d-i}(A)K_i$, we obtain $R^{d-2i+1}K_i=0$.
From this and since $XR=RX$, it follows that $R^{d-2i+1}XK_i=0$.  
By Definition \ref{def:K_i}, $K_i\subseteq U_i$ and hence $XK_i\subseteq U_{i-1}$.  
By Lemma \ref{lemma:Rkernel} the action of $R^{d-2i+1}$ on $U_{i-1}$ is injective.  
By these comments, $XK_i=0$.
We have now shown that $X=0$.
\end{proof}

\begin{lemma}
With reference to Lemma \ref{lemma:psidef'}, $\psi$ is the unique element of ${\rm End}(V)$ that satisfies both {\rm (\ref{psiR})}
and
$\psi U_i\subseteq U_{i-1}$ for $0\leq i\leq d$.
\end{lemma}
\begin{proof}
By Lemma \ref{lemma:psidef'} and Lemma \ref{lemma:psiU}, $\psi$ satisfies these conditions.  
We now show the uniqueness assertion.
Assume $\psi'\in {\rm End}(V)$ satisfies the conditions in the statement of the lemma.
Observe that $(\psi-\psi')R=R(\psi-\psi')$ and $(\psi-\psi')U_i\subseteq U_{i-1}$ for $0\leq i\leq d$.
The result follows from these comments along with Lemma \ref{lemma:X}.
\end{proof}

%%%%%%%%%%%%%%%%%%%
%%%%%%%%%%%%%%%%%%%
%%%%%%%%%%%%%%%%%%%
\section{The algebra $U_q(\mathfrak{sl}_2)$}\label{section:uqsl2}
%%%%%%%%%%%%%%%%%%%
%%%%%%%%%%%%%%%%%%%
%%%%%%%%%%%%%%%%%%%

In this section we recall the quantum universal enveloping algebra $U_q(\mathfrak{sl}_2)$.  See \cite{jantzen}, \cite{kassel} for background information.

\begin{definition}\label{def:uqsl2}{\rm
Let $U_q(\mathfrak{sl}_2)$ denote the $\K$-algebra with generators
$e, f, k, k^{-1}$ and relations
\begin{align}
kk^{-1}&=k^{-1}k=1,\nonumber\\
kek^{-1}=q^2e,&\qquad\qquad
kfk^{-1} =q^{-2}f,\label{eq:def-fk}\\
ef-fe&=\frac{k-k^{-1}}{q-q^{-1}}.\label{eq:def-ef}
\end{align}
We refer to $e,f,k^{\pm 1}$ as the {\it Chevalley generators} for $U_q(\mathfrak{sl}_2)$.
}\end{definition}

Following \cite[p. 21]{jantzen}, we 
define the normalized Casimir element $\Lambda$ for $U_q(\mathfrak{sl}_2)$ by
\begin{eqnarray}
\Lambda&=(q-q^{-1})^{2} ef+q^{-1}k+qk^{-1},\label{eq:casimirdef1}\\
&=(q-q^{-1})^{2} fe+qk+q^{-1}k^{-1}.\label{eq:casimirdef2}
\end{eqnarray}
By \cite[Lemma 2.7]{jantzen}, $\Lambda$ is central in $U_q(\mathfrak{sl}_2)$.

\begin{lemma}\label{lemma:f^2e}
With reference to Definition \ref{def:uqsl2}, 
both
\begin{align}
f^2e-(q^2+q^{-2})fef+ef^2 &=-\Lambda f,\label{eq:f^2e}\\
e^2 f-(q^2+q^{-2})efe + fe^2 &=-\Lambda e.\label{eq:e^2f}
\end{align}
\end{lemma}
\begin{proof}
We first prove (\ref{eq:f^2e}).  
The left-hand side of (\ref{eq:f^2e}) is equal to
\begin{equation}
-f(ef-fe)-(q-q^{-1})^2fef +(ef-fe)f.\label{eq:f^2e-2}
\end{equation}
By (\ref{eq:def-ef}), the element (\ref{eq:f^2e-2}) is equal to
\begin{equation*}
-f\frac{k-k^{-1}}{q-q^{-1}}-(q-q^{-1})^{2}fef +\frac{k-k^{-1}}{q-q^{-1}}f.
\end{equation*}
Line (\ref{eq:f^2e}) follows from this along with (\ref{eq:def-fk}) and (\ref{eq:casimirdef2}).

The proof is similar for (\ref{eq:e^2f}).
\end{proof}

We now discuss the finite-dimensional modules for $U_q(\mathfrak{sl}_2)$.  Recall the natural numbers $\N = \{0,1,2,\mathellipsis\}$ and the integers $\Z = \{0,\pm 1,\pm 2,\mathellipsis\}$.  For $n\in\Z$ define
\begin{equation*}
[n]_q=\frac{q^n-q^{-n}}{q-q^{-1}}.
\end{equation*} 
For $n\in\N$ define
\begin{equation*}
[n]_q^!=[n]_q [n-1]_q\cdots [1]_q,
\end{equation*}
where we interpret $[0]_q^!=1$.

\begin{lemma}{\rm \cite[Theorem 2.6]{jantzen}}.\label{lemma:?}  
 For $n\in\N$ and $\varepsilon\in\{1,-1\}$, there exists a $U_q(\mathfrak{sl}_2)$-module $L(n,\varepsilon)$ with the following properties.  $L(n,\varepsilon)$ has a basis $\{v_i\}_{i=0}^n$ such that
\begin{align}
ev_i & =\varepsilon [n+1-i]_qv_{i-1}    &(1\leq i\leq n),\qquad\qquad &ev_0=0,\\
fv_i & =[i+1]_qv_{i+1}    &(0\leq i\leq n-1),\qquad\qquad &fv_n=0,\\
kv_i&=\varepsilon q^{n-2i}v_i    &(0\leq i\leq n).\qquad\qquad &
\end{align}
The $U_q(\mathfrak{sl}_2)$-module $L(n,\varepsilon)$ is irreducible provided that $q^{2i}\neq 1$ for $1\leq i\leq n$.  
\end{lemma}

With reference to Lemma \ref{lemma:?} we refer to $\varepsilon$ as the {\it type} of $L(n,\varepsilon)$.

\begin{lemma}{\rm \cite[Lemma 2.7]{jantzen}}.\label{lemma:cas}
With reference to Lemma \ref{lemma:?}, for $n\in \N$ and $\varepsilon\in\{1,-1\}$, $\Lambda$ acts  on $L(n,\varepsilon)$ as $\varepsilon(q^{n+1}+q^{-n-1})$ times the identity.
\end{lemma}

If $q$ is not a root of unity, then the $L(n,\varepsilon)$ ($n\in\N, \varepsilon\in\{ 1,-1\}$) are the only finite-dimensional irreducible $U_q(\mathfrak{sl}_2)$-modules.  If $q$ is a root of unity, there are other types of finite-dimensional irreducible $U_q(\mathfrak{sl}_2)$-modules.  See \cite[Chapter 2]{jantzen} for a complete classification.  In our application, we will only be concerned with the finite-dimensional irreducible $U_q(\mathfrak{sl}_2)$-modules of type $L(n,\varepsilon)$.  \\

We now consider finite-dimensional $U_q(\mathfrak{sl}_2)$-modules which are not necessarily irreducible.

\begin{definition}{\rm
Let $V$ denote a finite-dimensional $U_q(\mathfrak{sl}_2)$-module.  We say that $V$ is {\it semisimple} whenever it is 
a direct sum of 
irreducible $U_q(\mathfrak{sl}_2)$-modules.}
\end{definition}

\begin{definition}{\rm \cite[Section 2.2]{jantzen}}.\label{def:wtspace}
{\rm 
Let $V$ denote a finite-dimensional $U_q(\mathfrak{sl}_2)$-module.  For $\lambda\in \K$, let $V_\lambda=\{v\in V | kv=\lambda v\}$.  
We call $\lambda$ a {\it weight} of $V$ whenever $V_\lambda\neq 0$.  
In this case we call $V_\lambda$ the {\it weight space of}  $V$ {\it associated with} $\lambda$.
}\end{definition}

Referring to Lemma \ref{lemma:?}, assume $q^{2i}\neq 1$ for $1\leq i\leq n$.
Observe that the weights of $L(n,\varepsilon)$ are $\varepsilon q^n,\varepsilon q^{n-2},\mathellipsis, \varepsilon q^{-n}$.  
We note that 
for $0\leq i\leq n$, $v_i$ is a basis for the weight space of $L(n,\varepsilon)$ associated with the weight $\varepsilon q^{n-2i}$.

\begin{definition} 
 {\rm  
 Let $V$ denote a finite-dimensional $U_q(\mathfrak{sl}_2)$-module.  Let $\lambda$ denote a weight of $V$.  By the {\it highest weight space of $V$ associated with} $\lambda$, 
 we mean the kernel of the action of $e$ on $V_{\lambda}$.  
We refer to $\lambda$ as a {\it highest weight} of $V$ whenever the corresponding highest weight space is nonzero.  
}\end{definition}

Referring to Lemma \ref{lemma:?}, assume $q^{2i}\neq 1$ for $1\leq i\leq n$.  
We note that $\varepsilon q^n$ is the unique highest weight of $L(n,\varepsilon)$.  For $L(n,\varepsilon)$, the highest weight space associated with the weight $\varepsilon q^n$ is equal to the weight space associated with $\varepsilon q^n$.

\begin{definition}\label{def:homcomp}
{\rm 
Let $V$ denote a finite-dimensional $U_q(\mathfrak{sl}_2)$-module.  For $n\in\N$ and $\varepsilon\in\{1,-1\}$, 
consider the subspace of $V$ spanned by the $U_q(\mathfrak{sl}_2)$-submodules of $V$ which are isomorphic to $L(n,\varepsilon)$.  We call this subspace the {\it homogeneous component of} $V$ {\it associated with} $L(n, \varepsilon)$.
}\end{definition}

%%%%%%%%%%%%%%%%%%%%%%%%
%%%%%%%%%%%%%%%%%%%%%%%%
%%%%%%%%%%%%%%%%%%%%%%%%
\section{A $U_q(\mathfrak{sl}_2)$-module structure on $V$ associated with $\Phi$}\label{section:Umod}
%%%%%%%%%%%%%%%%%%%%%%%%
%%%%%%%%%%%%%%%%%%%%%%%%
%%%%%%%%%%%%%%%%%%%%%%%%

We now return to the situation of Assumption \ref{assump:main}.   
Recall from Lemma \ref{lemma:psidef'} the equation
\begin{equation}
\psi R-R\psi=(q-q^{-1})(K-K^{-1}).\label{eq:psiR-secUmod}
\end{equation}
Recall from Lemma \ref{lemma:KRKinv} and Lemma \ref{lemma:KpsiKinv} that 
\begin{equation}
KRK^{-1}=q^{-2}R,\qquad\qquad K\psi K^{-1}=q^2\psi. \label{eq:KR,Kpsi}
\end{equation}
These relations are reminiscent of the defining relations for $U_q(\mathfrak{sl}_2)$.  
In this section we use the above relations to obtain a $U_q(\mathfrak{sl}_2)$-module structure on $V$.  
Then we will discuss this $U_q(\mathfrak{sl}_2)$-module structure from various points of view.

\begin{lemma}\label{lemma:Uaction}
With reference to Definition \ref{def:uqsl2}, there exists a $U_q(\mathfrak{sl}_2)$-module structure on $V$ for which the Chevalley generators act as follows:
\begin{center}
\begin{tabular}{c|cccc}
{\rm element of $U_q(\mathfrak{sl}_2)$} &  $e$ & $f$ &$k$ & $k^{-1}$
\\
\hline
{\rm action on $V$} & $(q-q^{-1})^{-1}\psi$ \ \ & $(q-q^{-1})^{-1}R$ \ \ &$K$ \ & $K^{-1}$
\end{tabular}
\end{center}
\end{lemma}
\begin{proof}
Use (\ref{eq:psiR-secUmod}), (\ref{eq:KR,Kpsi}), and Definition \ref{def:uqsl2}.
\end{proof}
\bigskip

Recall the Casimir element $\Lambda$ of $U_q(\mathfrak{sl}_2)$ from (\ref{eq:casimirdef1}), (\ref{eq:casimirdef2}).  

\begin{lemma}\label{lemma:casimir-act1} 
The action of $\Lambda$ on $V$ is equal to both
\begin{eqnarray}
&\psi R+q^{-1}K+qK^{-1},\label{eq:casimir-act1-1}\\
&R\psi+qK+q^{-1}K^{-1}.\label{eq:casimir-act1-2}
\end{eqnarray}
\end{lemma}
\begin{proof}
Use (\ref{eq:casimirdef1}), (\ref{eq:casimirdef2}), and Lemma \ref{lemma:Uaction}.
\end{proof}

\begin{lemma}
The action of $\Lambda$ on $V$ commutes with each of
\begin{equation*}
 \psi,\qquad\qquad R, \qquad\qquad K,\qquad\qquad A.\end{equation*}
\end{lemma}
\begin{proof}
Since $\Lambda$ is central in $U_q(\mathfrak{sl}_2)$, $\Lambda$ commutes with each of $e,f,k$.  
So the action of $\Lambda$ on $V$ commutes with each of $\psi, R, K$  in view of Lemma \ref{lemma:Uaction}.  The action of $\Lambda$ on $V$ commutes with $A$ by (\ref{eq:RAK}). 
\end{proof}

\begin{lemma}\label{lemma:R^2psi} 
The following equations hold on $V$:
\begin{eqnarray}
R^2\psi-(q^2+q^{-2})R\psi R+\psi R^2&=-(q-q^{-1})^2\Lambda R,\label{eq:R^2psi}\\
\psi^2 R-(q^2+q^{-2})\psi R\psi + R\psi^2&=-(q-q^{-1})^2\Lambda\psi.\label{eq:psi^2R}
\end{eqnarray}
\end{lemma}
\begin{proof}
Use Lemma \ref{lemma:f^2e} and Lemma \ref{lemma:Uaction}.
\end{proof}

\begin{lemma}\label{lemma:Uwtspace}
For $0\leq i\leq d$, $U_i$ is the weight space of the $U_q(\mathfrak{sl}_2)$-module $V$ associated with the weight $q^{d-2i}$.
\end{lemma}
\begin{proof}
Recall from Definition \ref{def:KB} that $U_i$ is an eigenspace of $K$ with corresponding eigenvalue $q^{d-2i}$.  The result follows.
\end{proof}

\begin{cor}
The weights of the $U_q(\mathfrak{sl}_2)$-module $V$ are $q^d,q^{d-2},\mathellipsis,q^{-d}$.  
\end{cor}

\begin{lemma}\label{lemma:Ki-hw}
For $0\leq i\leq d/2$, $K_i$ is the highest weight space of the $U_q(\mathfrak{sl}_2)$-module $V$ associated with the weight $q^{d-2i}$.
\end{lemma}
\begin{proof}
Use Lemma \ref{lemma:psiUkernel} and Lemma \ref{lemma:Uwtspace}.
\end{proof}

\begin{lemma}\label{lemma:Mv-mod}
Let $0\leq i\leq d/2$ and $0\neq v\in K_i$.  Then $Mv$ is an irreducible $U_q(\mathfrak{sl}_2)$-submodule of $V$.  The $U_q(\mathfrak{sl}_2)$-module $Mv$ is isomorphic to $L(d-2i,1)$.
\end{lemma}
\begin{proof}
For $0\leq j\leq d-2i$, let  
$v_j=\gamma_j^{-1} \tau_{i,i+j}(A)v$, where $\gamma_j=(q-q^{-1})^{j}[j]_q^!$.  By Corollary \ref{cor:Mv1}, $\{v_j\}_{j=0}^{d-2i}$ is a basis for $Mv$.  
By Lemma \ref{lemma:RonU} and Lemma \ref{lemma:refinement}, $Rv_j=(q-q^{-1})[j+1]_qv_{j+1}$ for $0\leq j\leq d-2i-1$ and $Rv_{d-2i}=0$.  By (\ref{eq:psiclar}), $\psi v_0=0$ and $\psi v_j=(q-q^{-1})[d-2i+1-j]_q v_{j-1}$ for $1\leq j\leq d-2i$.  By Lemma \ref{lemma:refinement}, $Kv_j=q^{d-2i-2j}v_j$ for $0\leq j\leq d-2i$.  The result follows from the above comments along with Lemma \ref{lemma:?} and Lemma \ref{lemma:Uaction}.
\end{proof}

\begin{lemma}\label{lemma:MKi-hom}
For $0\leq i\leq d/2$, $MK_i$ is a $U_q(\mathfrak{sl}_2)$-submodule of $V$.  Moreover $MK_i$ is the homogeneous component of $V$ associated with $L(d-2i,1)$.
\end{lemma}
\begin{proof}
Use Proposition \ref{prop:VMK_i} and Lemma \ref{lemma:Mv-mod}. 
\end{proof}

\begin{lemma}\label{lemma:Ucasaction}
For $0\leq i\leq d/2$, $MK_i$ is an eigenspace for $\Lambda$ with corresponding eigenvalue $q^{d-2i+1}+q^{2i-d-1}$. 
\end{lemma}
\begin{proof}
By Lemma \ref{note:main}, the scalars $\{q^{d-2j+1}+q^{ 2j-d-1}\}_{j=0}^{\lfloor d/2\rfloor}$ are mutually distinct.  The result follows from this along with Proposition \ref{prop:VMK_i}, Lemma \ref{lemma:cas}, and Lemma \ref{lemma:MKi-hom}.
\end{proof}

\begin{lemma}
The $U_q(\mathfrak{sl}_2)$-module $V$ is semisimple.  
Let $W$ denote an irreducible $U_q(\mathfrak{sl}_2)$-submodule of $V$.  Then there exists an integer $i$ $(0\leq i\leq d/2)$ such that $W$ is isomorphic to $L(d-2i,1)$.
\end{lemma}
\begin{proof}
Use Proposition \ref{prop:VMK_i}, Lemma \ref{lemma:Mv-mod}, and Lemma \ref{lemma:MKi-hom}.
\end{proof}

%%%%%%%%%%%%%%%%%%%%%%%%
%%%%%%%%%%%%%%%%%%%%%%%%
\section{A $U_q(\mathfrak{sl}_2)$-module structure on $V$ associated with $\Phi^{\Downarrow}$}\label{section:Umoddd}
%%%%%%%%%%%%%%%%%%%%%%%%
%%%%%%%%%%%%%%%%%%%%%%%%

We continue to discuss the situation of Assumption \ref{assump:main}.  
In Section \ref{section:Umod} we used $\Phi$ to obtain a $U_q(\mathfrak{sl}_2)$-module structure on $V$.  In the present section, we consider the corresponding $U_q(\mathfrak{sl}_2)$-module structure on $V$ associated with $\Phi^\Downarrow$.\\  

Recall from Lemma \ref{lemma:psiequal} that $\psi^\Downarrow=\psi$. Applying Lemma  \ref{lemma:psidef'} to $\Phi^\Downarrow$ we obtain
\begin{equation}
\psi R^\Downarrow-R^\Downarrow \psi =(q-q^{-1})(B-B^{-1}).\label{eq:psiRdd--1}
\end{equation}
Recall from Lemma \ref{lemma:KRKinv} and Lemma \ref{lemma:KpsiKinv} that 
\begin{equation}
B\psi B^{-1}=q^2\psi,\qquad\qquad
BR^\Downarrow B^{-1}=q^{-2}R^\Downarrow.\label{eq:KR,Kpsidd-secUdd}
\end{equation}

\begin{lemma}\label{lemma:Uddaction}
With reference to Definition \ref{def:uqsl2}, there exists a $U_q(\mathfrak{sl}_2)$-module structure on $V$ for which the Chevalley generators act as follows:
\begin{center}
\begin{tabular}{c|cccc}
{\rm element of $U_q(\mathfrak{sl}_2)$} &  $e$ & $f$ & $k$ & $k^{-1}$
\\
\hline
{\rm action on $V$} & $(q-q^{-1})^{-1}\psi$ \ \ & $(q-q^{-1})^{-1}R^\Downarrow$ \ \ & $B$  \ & $B^{-1}$
\end{tabular}
\end{center}
\end{lemma}
\begin{proof}
Use (\ref{eq:psiRdd--1}), (\ref{eq:KR,Kpsidd-secUdd}), and Definition \ref{def:uqsl2}.
\end{proof}
\bigskip

For the rest of this section, we will discuss the $U_q(\mathfrak{sl}_2)$-module $V$ from Lemma \ref{lemma:Uddaction}.\\
  
Recall the Casimir element $\Lambda$.  

\begin{lemma}\label{lemma:casimir-act2}
The action of $\Lambda$ on $V$ is equal to both
\begin{eqnarray}
&\psi R^\Downarrow+q^{-1}B+qB^{-1},\label{eq:casimir-act2-1}\\
&R^\Downarrow \psi+qB+q^{-1}B^{-1}.\label{eq:casimir-act2-2}
\end{eqnarray}
\end{lemma}

\begin{lemma}
The action of $\Lambda$ on $V$ commutes with each of
\begin{equation*}
 \psi,\qquad\qquad R^\Downarrow, \qquad\qquad B,\qquad\qquad A.
\end{equation*}
\end{lemma}

\begin{lemma}\label{lemma:R^2psidd}
The following equations hold on $V$:
\begin{align}
(R^{\Downarrow})^2\psi-(q^2+q^{-2})R^{\Downarrow}\psi R^{\Downarrow}+\psi (R^{\Downarrow})^2&=-(q-q^{-1})^2\Lambda R^{\Downarrow},\label{eq:R^2psidd}\\
\psi^2 R^{\Downarrow}-(q^2+q^{-2})\psi R^{\Downarrow}\psi + R^{\Downarrow}\psi^2&=-(q-q^{-1})^2\Lambda\psi.\label{eq:psi^2Rdd}
\end{align}
\end{lemma}

\begin{lemma}\label{lemma:Uddwtspace}
For $0\leq i\leq d$, $U_i^\Downarrow$ is the weight space of the $U_q(\mathfrak{sl}_2)$-module $V$ associated with the weight $q^{d-2i}$.
\end{lemma}

\begin{cor}
The weights of the $U_q(\mathfrak{sl}_2)$-module $V$ are $q^d,q^{d-2},\mathellipsis,q^{-d}$.  
\end{cor}

\begin{lemma}
For $0\leq i\leq d/2$, $K_i$ is the highest weight space of the $U_q(\mathfrak{sl}_2)$-module $V$ associated with the weight $q^{d-2i}$.
\end{lemma}

\begin{lemma}
Let $0\leq i\leq d/2$ and $0\neq v\in K_i$.  Then $Mv$ is an irreducible $U_q(\mathfrak{sl}_2)$-submodule of $V$.  The $U_q(\mathfrak{sl}_2)$-module $Mv$ is isomorphic to $L(d-2i,1)$.
\end{lemma}

\begin{lemma}
For $0\leq i\leq d/2$, $MK_i$ is a $U_q(\mathfrak{sl}_2)$-submodule of $V$.  Moreover $MK_i$ is the homogeneous component of $V$ associated with $L(d-2i,1)$.
\end{lemma}

\begin{lemma}\label{lemma:Uddcasaction}
For $0\leq i\leq d/2$, $MK_i$ is an eigenspace for $\Lambda$ with corresponding eigenvalue $q^{d-2i+1}+q^{ 2i-d-1}$.
\end{lemma}

\begin{lemma}
The $U_q(\mathfrak{sl}_2)$-module $V$ is semisimple.  
Let $W$ denote an irreducible $U_q(\mathfrak{sl}_2)$-submodule of $V$.  Then there exists an integer $i$ $(0\leq i\leq d/2)$ such that $W$ is isomorphic to $L(d-2i,1)$. 
\end{lemma}

\begin{note}\label{note:uqsl2hat}
{\rm 
The paper \cite{IT:qRacah} describes an action of $U_q(\widehat{\mathfrak{sl}}\sb 2)$ on $V$.
Roughly speaking, $U_q(\widehat{\mathfrak{sl}}\sb 2)$ is generated by two copies of $U_q(\mathfrak{sl}_2)$ that are glued together in a certain way \cite[p. 262]{chari&pressley}.  Thus the action of $U_q(\widehat{\mathfrak{sl}}\sb 2)$ on $V$ induces two actions of $U_q(\mathfrak{sl}_2)$ on $V$.  For these actions the Chevalley generator $e$ does not act as a scalar multiple of $\psi$ {\rm \cite[Lines (28) and (30)]{IT:qRacah}}.  
Therefore the two $U_q(\mathfrak{sl}_2)$-actions from \cite{IT:qRacah} are not the same as the two $U_q(\mathfrak{sl}_2)$-actions from Lemma \ref{lemma:Uaction} and Lemma \ref{lemma:Uddaction}.   
As far as we know, the two $U_q(\mathfrak{sl}_2)$-actions from \cite{IT:qRacah} are not directly related to the  
$U_q(\mathfrak{sl}_2)$-actions from Lemma \ref{lemma:Uaction} and Lemma \ref{lemma:Uddaction}.}
\end{note}

%%%%%%%%%%%%%%%%%%
%%%%%%%%%%%%%%%%%%
\section{How $\psi, K^{\pm 1},B^{\pm 1}$ are related}\label{section:related}
%%%%%%%%%%%%%%%%%%
%%%%%%%%%%%%%%%%%%

We continue to discuss the situation of Assumption \ref{assump:main}.  
In Sections \ref{section:Umod} and \ref{section:Umoddd} we introduced two $U_q(\mathfrak{sl}_2)$-module structures on $V$.  
In this section we compare these module structures.  
From this comparison, we obtain several equations relating $\psi, K^{\pm 1}, B^{\pm 1}$.\\

Recall the Casimir element $\Lambda$ of $U_q(\mathfrak{sl}_2)$ from Section \ref{section:uqsl2}.

\begin{lemma}\label{lemma:4exp}
The following coincide:
\begin{itemize}
\item[{\rm (i)}] the action of $\Lambda$ on $V$ for the $U_q(\mathfrak{sl}_2)$-module structure from Lemma \ref{lemma:Uaction},
\item[{\rm (ii)}] the action of $\Lambda$ on $V$ for the $U_q(\mathfrak{sl}_2)$-module structure from Lemma \ref{lemma:Uddaction}.
\end{itemize}
\end{lemma}
\begin{proof}
Use Proposition \ref{prop:VMK_i}, Lemma \ref{lemma:Ucasaction}, and Lemma \ref{lemma:Uddcasaction}.
\end{proof}

\begin{prop}\label{prop:coincide}
The following coincide:
\begin{align*}
(I-aq\psi)K, \qquad (I-a^{-1}q\psi)B, \qquad K(I-aq^{-1}\psi), \qquad B(1-a^{-1}q^{-1}\psi).
\end{align*}
Moreover the following coincide:
\begin{align*}
(I-a^{-1}q^{-1}\psi)K^{-1}, \qquad (I-aq^{-1}\psi)B^{-1}, \qquad K^{-1}(I-a^{-1}q\psi), \qquad B^{-1}(1-aq\psi).
\end{align*}
\end{prop}
\begin{proof}
Consider the expression $q^{-1}$ times (\ref{eq:casimir-act1-1}) minus $q$ times (\ref{eq:casimir-act1-2}) minus $q^{-1}$ times (\ref{eq:casimir-act2-1}) plus $q$ times (\ref{eq:casimir-act2-2}).  We evaluate this expression in two ways.  First, by Lemma \ref{lemma:4exp} this expression is equal to zero.  Second, eliminate $R$ and $R^\Downarrow$ using (\ref{eq:RAK}) and (\ref{eq:RAKdd}) and simplify the result using Lemma \ref{lemma:KpsiKinv}.  By these comments,
\begin{equation}
(1-aq\psi)K=(1-qa^{-1}\psi)B.\label{eq:coincidework-1}
\end{equation}
The remaining assertions follow from (\ref{eq:coincidework-1}) and Lemma \ref{lemma:KpsiKinv}.
\end{proof}

Shortly we will write $KB^{-1},$ $K^{-1}B,$ and their inverses in terms of $\psi$.  In order to do this, we will need that certain elements of ${\rm End}(V)$ are invertible.

\begin{lemma}\label{lemma:invertible2}
Each of the following is invertible:
\begin{equation}
I-aq\psi,\qquad I-a^{-1}q\psi, \qquad I-aq^{-1}\psi,\qquad I-a^{-1}q^{-1}\psi.\label{eq:invertible2}
\end{equation}
Their inverses are as follows:
\begin{align}
&&(I-aq\psi)^{-1}&=\sum_{i=0}^da^iq^i\psi^i,\qquad &(I-a^{-1}q\psi)^{-1}&=\sum_{i=0}^d a^{-i}q^i\psi^i,\label{eq:psiinv1'}\\
&&(I-aq^{-1}\psi)^{-1}&=\sum_{i=0}^d a^iq^{-i}\psi^i,\qquad &(I-a^{-1}q^{-1}\psi)^{-1}&=\sum_{i=0}^d a^{-i}q^{-i}\psi^i.\label{eq:psiinv4'}
\end{align}
\end{lemma}
\begin{proof}
Recall from Lemma \ref{lemma:psiU} that $\psi^{d+1}=0$.  
\end{proof}

\begin{thm}\label{thm:BK}
The following hold:
\begin{align}
BK^{-1}&=\frac{I-aq\psi}{I-a^{-1}q\psi}, \qquad &&KB^{-1}=\frac{I-a^{-1}q\psi}{I-aq\psi},\label{eq:BK-1}\\
K^{-1}B&=\frac{I-aq^{-1}\psi}{I-a^{-1}q^{-1}\psi}, \qquad  &&B^{-1}K=\frac{I-a^{-1}q^{-1}\psi}{I-aq^{-1}\psi}.\label{eq:BK-2}
\end{align}
In {\rm (\ref{eq:BK-1})}, {\rm (\ref{eq:BK-2})} the denominators are invertible by Lemma \ref{lemma:invertible2}.
\end{thm}
\begin{proof}
Use Proposition \ref{prop:coincide} and Lemma \ref{lemma:invertible2}.
\end{proof}

\begin{lemma}\label{lemma:KBcomm}
The following mutually commute:
\begin{equation*}
\psi, \qquad\qquad BK^{-1},\qquad\qquad KB^{-1},\qquad\qquad K^{-1}B,\qquad\qquad B^{-1}K.
\end{equation*}
\end{lemma}
\begin{proof}
By Theorem \ref{thm:BK} each of 
the four expressions on the right 
is a polynomial in $\psi$.
\end{proof}

Shortly we will give four ways to write $\psi$ in terms of $K,B$.  
In order to do this, we will need that certain elements of ${\rm End}(V)$ are invertible.

\begin{lemma}\label{lemma:I-KB}
Each of
\begin{equation}
I-BK^{-1},\qquad I-KB^{-1},\qquad I-K^{-1}B,\qquad I-B^{-1}K \label{eq:I-BK}
\end{equation}
sends $U_i$ into $U_0+U_1+\cdots+U_{i-1}$ for $0\leq i\leq d$.  Moreover each of {\rm (\ref{eq:I-BK})}  
is nilpotent.  
\end{lemma}
\begin{proof}
We first consider $I-BK^{-1}$.  
On $U_i$,
\begin{equation*}
I-BK^{-1}=I-q^{2i-d}B.
\end{equation*}
By this and Lemma \ref{lemma:KUdd}, $I-BK^{-1}$ sends $U_i$ into $U_0+U_1+\cdots+U_{i-1}$.

Since $KB^{-1}$ and $BK^{-1}$ are inverses, we see that $I-KB^{-1}$ sends $U_i$ into $U_0+U_1+\cdots+U_{i-1}$.

The proof is similar for $I-K^{-1}B$ and $I-B^{-1}K$.
\end{proof}

\begin{lemma}\label{lemma:invertible1}
Each of the following is invertible:
\begin{align*}
&aI-a^{-1}BK^{-1}, &a^{-1}I-aKB^{-1},\\
&aI-a^{-1}K^{-1}B, &a^{-1}I-aB^{-1}K.
\end{align*}
\end{lemma}
\begin{proof}
We show that $aI-a^{-1}BK^{-1}$ is invertible. 
Observe that 
\begin{equation*}
aI-a^{-1}BK^{-1}=(a-a^{-1})I+a^{-1}(I-BK^{-1}).
\end{equation*}
Now $aI-a^{-1}BK^{-1}$ is invertible by Lemma \ref{lemma:I-KB} and the fact that $a^{2}\neq 1$.

The remaining assertions are similarly proved.
\end{proof}

\begin{thm}\label{thm:psiequations}
The map $\psi$ is equal to each of the following:
\begin{eqnarray}
&{\displaystyle \frac{I-BK^{-1}}{q(aI-a^{-1}BK^{-1})}, \qquad\qquad\qquad \frac{I-KB^{-1}}{q(a^{-1}I-aKB^{-1})}},\label{eq:psiequations1}\\ \medskip
&{\displaystyle \frac{q(I-K^{-1}B)}{aI-a^{-1}K^{-1}B},\qquad\qquad\qquad \frac{q(I-B^{-1}K)}{a^{-1}I-aB^{-1}K}}.\label{eq:psiequations2}
\end{eqnarray}
In {\rm (\ref{eq:psiequations1})}, {\rm (\ref{eq:psiequations2})} the denominators are invertible by Lemma \ref{lemma:invertible1}.
\end{thm}
\begin{proof}
In each equation of Theorem \ref{thm:BK}, solve for $\psi$.
\end{proof}

\begin{thm}\label{thm:KBquad}
We have
\begin{equation}
aK^2-\frac{a^{-1}q-aq^{-1}}{q-q^{-1}}\ KB-\frac{aq-a^{-1}q^{-1}}{q-q^{-1}}\ BK+a^{-1}B^2=0.\label{eq:KBquad1}
\end{equation}
\end{thm}
\begin{proof}
Equate the expression on the left in (\ref{eq:psiequations1}) and the expression on the right in (\ref{eq:psiequations2}).  
For every term in the resulting equation, multiply on the left by $B(a^{-1}I-aB^{-1}K)$ and on the right by $(aI-a^{-1}BK^{-1})K$. 
\end{proof}\\

We mention a reformulation of Theorem \ref{thm:KBquad}.  

\begin{thm}\label{thm:KBinvquad}
We have 
\begin{equation}
aB^{-2}-\frac{a^{-1}q-aq^{-1}}{q-q^{-1}}\ K^{-1}B^{-1}-\frac{aq-a^{-1}q^{-1}}{q-q^{-1}}\ B^{-1}K^{-1}+a^{-1}K^{-2}=0.\label{eq:KBquad2}
\end{equation}
\end{thm}
\begin{proof}
For every term in (\ref{eq:KBquad1}), multiply on the left by $B^{-1}K^{-1}$ 
and on the right by $K^{-1}B^{-1}$. 
Simplify the result using Lemma \ref{lemma:KBcomm}.  
\end{proof}
\bigskip

Equations (\ref{eq:KBquad1}) and (\ref{eq:KBquad2}) can be put in the following attractive forms.
\begin{lemma}\label{lemma:KBfactor} 
The following equations hold:
\begin{align}
q(K-B)(aK-a^{-1}B)&=q^{-1}(aK-a^{-1}B)(K-B),\label{eq:qcomm-1}\\
q(a^{-1}K^{-1}-aB^{-1})(K^{-1}-B^{-1})&=q^{-1}(K^{-1}-B^{-1})(a^{-1}K^{-1}-aB^{-1}),\label{eq:qcomm-2}\\
q(I-K^{-1}B)(aI-a^{-1}BK^{-1})&=q^{-1}(aI-a^{-1}K^{-1}B)(I-BK^{-1}),\label{eq:qcomm-3}\\
q(a^{-1}I-aKB^{-1})(I-B^{-1}K)&=q^{-1}(I-KB^{-1})(a^{-1}I-aB^{-1}K).\label{eq:qcomm-4}
\end{align}
\end{lemma}
\begin{proof}
To verify (\ref{eq:qcomm-1}), multiply out each side and compare the result with (\ref{eq:KBquad1}).  
Equation (\ref{eq:qcomm-2}) is similarly verified using (\ref{eq:KBquad2}).  
To verify (\ref{eq:qcomm-3}), multiply each term in (\ref{eq:qcomm-1}) on the left by $K^{-1}$ and on the right by $K^{-1}$.  
To verify (\ref{eq:qcomm-4}), 
multiply each term in (\ref{eq:qcomm-2}) on the left by $K$ and on the right by $K$.  
\end{proof}

%%%%%%%%%%%%%%%%%%%%%
%%%%%%%%%%%%%%%%%%%%%
\section{How $R, K^{\pm 1}$ and $R^\Downarrow, B^{\pm 1}$ are related}\label{section:related2}
%%%%%%%%%%%%%%%%%%%%%
%%%%%%%%%%%%%%%%%%%%%

We continue to discuss the situation of Assumption \ref{assump:main}. 
In Sections \ref{section:Umod} and \ref{section:Umoddd} we displayed two $U_q(\mathfrak{sl}_2)$-actions on $V$.  
A natural question is, can we write each of the operators for one action in terms of the operators for the other action.  In this section we demonstrate that this can be done.\\

Recall from Lemma \ref{lemma:psiequal} that
\begin{equation*}
\psi=\psi^\Downarrow.
\end{equation*}

We now give $R^\Downarrow, B^{\pm 1}$ in terms of $\psi, R, K^{\pm 1}$.

\begin{lemma}\label{lemma:KKBB1}
The following equations hold:
\begin{align}
B&=a^{2}K+(1-a^{2})K\sum_{i=0}^d a^{-i}q^{-i}\psi^i,\label{eq:KKBB1-1}\\
B^{-1}&=a^{-2}K^{-1}+(1-a^{-2})K^{-1}\sum_{i=0}^d a^{i}q^{i}\psi^i,\label{eq:KKBB1-2}\\
R^\Downarrow &=R+(a-a^{-1})\sum_{i=0}^d (a^{-i}q^{-i}K-a^iq^{i}K^{-1})\psi^i.\label{eq:KKBB1-3}
\end{align}
\end{lemma}
\begin{proof}
To obtain (\ref{eq:KKBB1-1}) and (\ref{eq:KKBB1-2}), use Theorem \ref{thm:BK} along with (\ref{eq:psiinv1'}) and (\ref{eq:psiinv4'}).
To obtain (\ref{eq:KKBB1-3}), use (\ref{eq:Rdiff}) along with (\ref{eq:KKBB1-1}) and (\ref{eq:KKBB1-2}).
\end{proof}\\

We now give $R, K^{\pm 1}$ in terms of $\psi, R^\Downarrow, B^{\pm 1}$.

\begin{lemma}\label{lemma:KKBB2}
The following equations hold:
\begin{align}
K&=a^{-2}B+(1-a^{-2})B\sum_{i=0}^d a^iq^{-i}\psi^i,\label{eq:KKBB2-1}\\
K^{-1}&=a^{2}B^{-1}+(1-a^{2})B^{-1} \sum_{i=0}^d a^{-i}q^{i}\psi^i,\label{eq:KKBB2-2}\\
R&=R^\Downarrow+(a-a^{-1})\sum_{i=0}^d (a^{-i}q^iB^{-1}-a^iq^{-i}B)\psi^i.\label{eq:KKBB2-3}
\end{align}
\end{lemma}
\begin{proof}
To obtain (\ref{eq:KKBB2-1}) and (\ref{eq:KKBB2-2}), use (\ref{eq:KKBB1-1}) and (\ref{eq:KKBB1-2}) along with Lemma \ref{lemma:KpsiKinv}.  
To obtain (\ref{eq:KKBB2-3}), use (\ref{eq:Rdiff}) along with (\ref{eq:KKBB2-1}) and (\ref{eq:KKBB2-2}).
\end{proof}

%%%%%%%%%%%%%%%%%%%%%
%%%%%%%%%%%%%%%%%%%%%
\section{How $A, \psi$ are related}\label{section:Apsi}
%%%%%%%%%%%%%%%%%%%%%
%%%%%%%%%%%%%%%%%%%%%

We continue to discuss the situation of Assumption \ref{assump:main}.  
In this section we show how $A$ and $\psi$ are related.  
In what follows we refer to the $\Lambda$-action from Lemma \ref{lemma:4exp}.

\begin{lemma}
On $V$, we have
\begin{equation}
\begin{split}
A^2\psi-(q^2+q^{-2})&A\psi A+\psi A^2 +(q^2-q^{-2})^2\psi\\
&=-(q-q^{-1})^2\Lambda A+(a+a^{-1})(q-q^{-1})^2(q+q^{-1})I\label{eq:A^2psi}\\
\end{split}
\end{equation}
and also
\begin{equation}
\psi^2 A-(q^2+q^{-2})\psi A\psi + A\psi^2=-(q-q^{-1})^2\Lambda\psi.\label{eq:psi^2A}
\end{equation}
\end{lemma}
\begin{proof}
We first prove (\ref{eq:A^2psi}).  Let $L$ denote the expression on the left-hand side in (\ref{eq:A^2psi}).  
In $L$, eliminate $A$ using $A=R+aK+a^{-1}K^{-1}$.  Simplify the result using (\ref{eq:KR,Kpsi}) and (\ref{eq:R^2psi}).  This shows that
$L$ is equal to $-(q-q^{-1})^2\Lambda A$ plus $\Lambda-\psi R$ times
\begin{equation*}
a(q^{2}-1)K+a^{-1}(q^{-2}-1)K^{-1}
\end{equation*}
plus $\Lambda-R\psi$ times
\begin{equation*}
a(q^{-2}-1)K+a^{-1}(q^2-1)K^{-1}.
\end{equation*}
In this expression, eliminate $\Lambda-\psi R$ and $\Lambda-R\psi$ using Lemma \ref{lemma:casimir-act1}.  The resulting expression for $L$ is the right-hand side of (\ref{eq:A^2psi}).

The proof of (\ref{eq:psi^2A}) is similar.
\end{proof}

\section{Acknowledgments}

This paper was written while the author was a graduate student at the University of Wisconsin-Madison. The author would like to thank her advisor, Paul Terwilliger, for offering many valuable ideas and suggestions.\\

The author would also like to thank Kazumasa Nomura for giving this paper a close reading and offering many valuable suggestions.

%%%%%%%%%%%%%%%%%%%%%%%%%%%%%%%%%%%%%%%%%%%%%%%%%%%%%%

\noindent Sarah Bockting-Conrad \hfil\break
\noindent Department of Mathematics \hfil\break
\noindent University of Wisconsin \hfil\break
\noindent 480 Lincoln Drive \hfil\break
\noindent Madison, WI 53706-1388 USA \hfil\break
\noindent email: {\tt bockting@math.wisc.edu }\hfil\break

\end{document}